\documentclass[12pt,twoside]{amsart}
\usepackage{amsmath, amsthm, amscd, amsfonts, amssymb, graphicx}
\usepackage{enumerate}
\usepackage[colorlinks=true,
linkcolor=blue,
urlcolor=cyan,
citecolor=red]{hyperref}
\usepackage{mathrsfs}
\usepackage{fonttable}
\newtheorem{theorem}{Theorem}[section]
\newtheorem{lemma}{Lemma}[section]
\newtheorem{remark}{Remark}[section]
\newtheorem{question}{Question}[section]

\newtheorem{example}{Example}[section]

\numberwithin{equation}{section}

\begin{document}
\title{A note on some inequalities for positive linear maps}
\author{Hamid Reza Moradi$^1$, Mohsen Erfanian Omidvar$^2$, Ibrahim Halil G\"um\"u\c s$^3$ and Razieh Naseri$^4$}
\subjclass[2010]{47A63, 47A30.}
\keywords{Positive linear maps, operator norm, AM-GM inequality, Young inequality.} \maketitle
\begin{abstract}
We improve and generalize some operator inequalities for positive linear maps. It is shown, among other inequalities, that if $0<m\le B\le m'<M'\le A\le M$ or $0<m\le A\le m'<M'\le B\le M$, then for each $2\le p<\infty $ and $\nu \in \left[ 0,1 \right]$,
\begin{equation*}
{{\Phi }^{p}}\left( A{{\nabla }_{\nu }}B \right)\le {{\left( \frac{K\left( h \right)}{{{4}^{\frac{2}{p}-1}}{{K}^{r}}\left( h' \right)} \right)}^{p}}{{\Phi }^{p}}\left( A{{\#}_{\nu }}B \right),
\end{equation*} 
and
\begin{equation*}
{{\Phi }^{p}}\left( A{{\nabla }_{\nu }}B \right)\le {{\left( \frac{K\left( h \right)}{{{4}^{\frac{2}{p}-1}}{{K}^{r}}\left( h' \right)} \right)}^{p}}{{\left( \Phi \left( A \right){{\#}_{\nu }}\Phi \left( B \right) \right)}^{p}},
\end{equation*}
where $r=\min \left\{ \nu ,1-\nu \right\}$, $h=\frac{M}{m}$ and $h'=\frac{M'}{m'}$. We also obtain an improvement of operator P\'olya-Szeg\"o inequality.
\end{abstract}
\pagestyle{myheadings}
\markboth{\centerline {A note on some inequalities for positive linear maps}}
{\centerline {H.R. Moradi, M.E. Omidvar, I.H. G\"um\"u\c s \& R. Naseri}}
\bigskip
\bigskip
\section{Introduction}
Let us introduce a notation and state a few elementary facts that will be helpful in the ensuing discussion. Throughout this paper, we reserve $M,M',m,m'$ for real numbers and $I$ for the identity operator. Other capital letters denote general elements of the ${{C}^{*}}$-algebra $\mathcal{B}\left( \mathcal{H} \right)$ of all bounded linear operators on a complex separable Hilbert space $\left( \mathcal{H},\left\langle \cdot,\cdot \right\rangle \right)$. Also, we identify a scalar with the unit multiplied by this scalar. $\left\| \cdot \right\|$ denote the operator norm. An operator $A$ is said to be positive (strictly positive) if $\left\langle Ax,x \right\rangle \ge 0$ for all $x\in \mathcal{H}$ ($\left\langle Ax,x \right\rangle >0$ for all $x\in \mathcal{H}\backslash \left\{ 0 \right\}$) and write $A\ge 0$ ($A>0$). $A\ge B$ ($A>B$) means $A-B\ge 0$ ($A-B>0$). A linear map $\Phi :\mathcal{B}\left( \mathcal{H} \right)\to \mathcal{B}\left( \mathcal{K} \right)$ is called positive if $\Phi \left( A \right)\ge 0$ whenever $A\ge 0$. It is said to be unital if $\Phi \left( I \right)=I$.
As a matter of convenience, we use the following notations to define the weighted arithmetic and geometric means for operators:
\[A{{\nabla }_{\nu }}B=\left( 1-\nu  \right)A+\nu B,\quad A{{\#}_{\nu }}B={{A}^{\frac{1}{2}}}{{\left( {{A}^{-\frac{1}{2}}}B{{A}^{-\frac{1}{2}}} \right)}^{\nu }}{{A}^{\frac{1}{2}}},\]
where $A,B>0$ and $\nu \in \left[ 0,1 \right]$. 

Lin \cite{6} reduced the study of squared operator inequalities to that of some norm inequalities.

According to the celebrated paper by Lin {{\cite[Theorem 2.1]{6}}}, if $0<m\le A,B\le M$, then
\begin{equation}\label{14}
{{\Phi }^{2}}\left( \frac{A+B}{2} \right)\le {{K}^{2}}\left( h \right){{\Phi }^{2}}\left( A\#B \right),
\end{equation}
and 
\begin{equation}\label{15}
{{\Phi }^{2}}\left( \frac{A+B}{2} \right)\le {{K}^{2}}\left( h \right){{\left( \Phi \left( A \right)\#\Phi \left( B \right) \right)}^{2}},
\end{equation}
where $K\left( h \right)=\frac{{{\left( h+1 \right)}^{2}}}{4h}$ and $h=\frac{M}{m}$.

Related to this, Xue and Hu {{\cite[Theorem 2]{3}}} proved that if $0<m\le A\le m'<M'\le B\le M$, then
\begin{equation}\label{21}
{{\Phi }^{2}}\left( \frac{A+B}{2} \right)\le \frac{{{K}^{2}}\left( h \right)}{K\left( h' \right)}{{\Phi }^{2}}\left( A\#B \right),
\end{equation}
and
\begin{equation}\label{22}
{{\Phi }^{2}}\left( \frac{A+B}{2} \right)\le \frac{{{K}^{2}}\left( h \right)}{K\left( h' \right)}{{\left( \Phi \left( A \right)\#\Phi \left( B \right) \right)}^{2}}.
\end{equation}
where $K\left( h \right)=\frac{{{\left( h+1 \right)}^{2}}}{4h},\text{ }K\left( h' \right)=\frac{{{\left( h'+1 \right)}^{2}}}{4h'},\text{ }h=\frac{M}{m}$ and $h'=\frac{M'}{m'}$.
   
We will get a stronger result than \eqref{21}-\eqref{22} (see Theorem \ref{5}).

Lin's results was further generalized by several authors. Among them, Fu and He {{\cite[Theorem 4]{7}}} generalized \eqref{14} and \eqref{15} to the power of $p$ ($2\le p<\infty $) as follows:
\begin{equation}\label{16}
{{\Phi }^{p}}\left( \frac{A+B}{2} \right)\le {{\left( \frac{{{\left( M+m \right)}^{2}}}{{{4}^{\frac{2}{p}}}Mm} \right)}^{p}}{{\Phi }^{p}}\left( A\#B \right),
\end{equation}
and
\begin{equation}\label{17}
{{\Phi }^{p}}\left( \frac{A+B}{2} \right)\le {{\left( \frac{{{\left( M+m \right)}^{2}}}{{{4}^{\frac{2}{p}}}Mm} \right)}^{p}}{{\left( \Phi \left( A \right)\#\Phi \left( B \right) \right)}^{p}}.
\end{equation}
It is interesting to ask whether the inequalities \eqref{16} and \eqref{17} can be improved. This is an another motivation of the present paper (Theorem \ref{20}). We close the paper by improving operator P\'olya-Szeg\"o inequality (Theorem \ref{46}).
\section{Main Results}
We give some Lemmas before we give the main theorems of this paper:
\begin{lemma}\label{50}
{{\cite[Lemma 2.3]{ba}}} Let $A,B>0$ and $\alpha >0$, then
\[A\le \alpha B\quad\Leftrightarrow \quad\left\| {{A}^{\frac{1}{2}}}{{B}^{-\frac{1}{2}}} \right\|\le {{\alpha }^{\frac{1}{2}}}.\] 
\end{lemma}
\begin{lemma}\label{6}
{{\cite[Theorem 1]{4}}} Let $A,B>0$. Then the following norm inequality holds:
\[\left\| AB \right\|\le \frac{1}{4}{{\left\| A+B \right\|}^{2}}.\]
\end{lemma}
\vskip 0.4 true cm
We need the following inequality, which is due to Choi. It can be found, for example, in {{\cite[p. 41]{1}}}.
\begin{lemma}\label{7}
Let $A>0$. Then for every positive unital linear map $\Phi $,
\[{{\Phi }^{-1}}\left( A \right)\le \Phi \left( {{A}^{-1}} \right).\] 
\end{lemma}
\vskip 0.4 true cm
The following is attributed to Ando \cite{5}:
\begin{lemma}\label{11}
Let $\Phi $ be any (not necessary unital) positive linear map and $A,B$ be positive operators. Then for every $\nu \in \left[ 0,1 \right]$, 
\begin{equation*}
\Phi \left( A{{\#}_{\nu }}B \right)\le \Phi \left( A \right){{\#}_{\nu }}\Phi \left( B \right).
\end{equation*}
\end{lemma}
\vskip 0.4 true cm
The following basic lemma is essentially known as in {{\cite[Theorem 7]{2}}}, but our expression is a little bit different from those in \cite{2}. For the sake of convenience, we give it a slim proof.
\begin{lemma}\label{1}
Let $0<m\le B\le m'<M'\le A\le M$ or $0<m\le A\le m'<M'\le B\le M$, then 
\begin{equation}\label{33}
{{K}^{r}}\left( h' \right)\left( {{A}^{-1}}{{\#}_{\nu }}{{B}^{-1}} \right)\le {{A}^{-1}}{{\nabla }_{\nu }}{{B}^{-1}},
\end{equation}
for each $\nu \in \left[ 0,1 \right]$ and $r=\min \left\{ \nu ,1-\nu \right\}$, $h=\frac{M}{m}$ and $h'=\frac{M'}{m'}$. 
\end{lemma}
\begin{proof}
From {{\cite[Corollary 3]{2}}}, for any $a,b>0$ and $\nu \in \left[ 0,1 \right]$ we have
\begin{equation}\label{23}
{{K}^{r}}\left( h \right){{a}^{1-\nu }}{{b}^{\nu }}\le a{{\nabla }_{\nu }}b,
\end{equation}
where $h=\frac{b}{a}$ and $r=\min \left\{ \nu ,1-\nu \right\}$. 
Taking $a=1$ and $b=x$, then utilizing the continuous functional calculus and the fact that $0<h'I\le X\le hI$, we have
\[\underset{h'\le x\le h}{\mathop{\min }}\,{{K}^{r}}\left( x \right){{X}^{\nu }}\le \left( 1-\nu \right)I+\nu X,\] 
for any $x>0$. 
Replacing $X$ with $1<h'=\frac{M'}{m'}\le {{A}^{\frac{1}{2}}}{{B}^{-1}}{{A}^{\frac{1}{2}}}\le \frac{M}{m}=h$ 
\[\underset{h'\le x\le h}{\mathop{\min }}\,{{K}^{r}}\left( x \right){{\left( {{A}^{\frac{1}{2}}}{{B}^{-1}}{{A}^{\frac{1}{2}}} \right)}^{\nu }}\le \left( 1-\nu \right)I+\nu {{A}^{\frac{1}{2}}}{{B}^{-1}}{{A}^{\frac{1}{2}}}.\]
Since $K\left( x \right)$ is an increasing function for $x>1$, then 
\[{{K}^{r}}\left( h' \right){{\left( {{A}^{\frac{1}{2}}}{{B}^{-1}}{{A}^{\frac{1}{2}}} \right)}^{\nu }}\le \left( 1-\nu \right)I+\nu {{A}^{\frac{1}{2}}}{{B}^{-1}}{{A}^{\frac{1}{2}}}.\] 
Now by multiplying both sides by ${{A}^{-\frac{1}{2}}}$ we deduce the desired inequality \eqref{33}.

For the case of  $0<m\le A\le m'<M'\le B\le M$, since the function $K\left( x \right)$ is decreasing for $x<1$  and $K\left( x \right)=K\left( \frac{1}{x} \right)$, similarly we obtain inequality \eqref{33}.
\end{proof}
\vskip 0.4 true cm
Inequalities \eqref{21} and \eqref{22} can be generalized by means of weighted parameter $\nu \in \left[ 0,1 \right]$ as follows:
\begin{theorem}\label{5}
If $0<m\le B\le m'<M'\le A\le M$ or $0<m\le A\le m'<M'\le B\le M$, then for each $\nu \in \left[ 0,1 \right]$ we have
\begin{equation}\label{10}
{{\Phi }^{2}}\left( A{{\nabla }_{\nu }}B \right)\le {{\left( \frac{K\left( h \right)}{{{K}^{r}}\left( h' \right)} \right)}^{2}}{{\Phi }^{2}}\left( A{{\#}_{\nu }}B \right),
\end{equation}
and
\begin{equation}\label{12}
{{\Phi }^{2}}\left( A{{\nabla }_{\nu }}B \right)\le {{\left( \frac{K\left( h \right)}{{{K}^{r}}\left( h' \right)} \right)}^{2}}{{\left( \Phi \left( A \right){{\#}_{\nu }}\Phi \left( B \right) \right)}^{2}},
\end{equation}
where $r=\min \left\{ \nu ,1-\nu \right\}$, $K\left( h \right)=\frac{{{\left( h+1 \right)}^{2}}}{4h},K\left( h' \right)=\frac{{{\left( h'+1 \right)}^{2}}}{4h'},h=\frac{M}{m},h'=\frac{M'}{m'}$. 
\end{theorem}
\begin{proof}
According to the hypothesis we have
\[A+Mm{{A}^{-1}}\le M+m,\] 
and
\[B+Mm{{B}^{-1}}\le M+m.\] 
Also the following inequalities holds true
\begin{equation}\label{2}
\left( 1-\nu \right)A+\left( 1-\nu \right)Mm{{A}^{-1}}\le \left( 1-\nu \right)M+\left( 1-\nu \right)m,
\end{equation}
and
\begin{equation}\label{3}
\nu B+\nu Mm{{B}^{-1}}\le \nu M+\nu m.
\end{equation}
Now summing up \eqref{2} and \eqref{3} we obtain
\[A{{\nabla }_{\nu }}B+Mm\left( {{A}^{-1}}{{\nabla }_{\nu }}{{B}^{-1}} \right)\le M+m.\]
Applying positive linear map $\Phi $ we can write 
\begin{equation}\label{4}
\Phi \left( A{{\nabla }_{\nu }}B \right)+Mm\Phi \left( {{A}^{-1}}{{\nabla }_{\nu }}{{B}^{-1}} \right)\le M+m.
\end{equation}
With inequality \eqref{4} in hand, we are ready to prove \eqref{10}. 

By Lemma \ref{50}, it is enough to prove that 
\[\left\| \Phi \left( A{{\nabla }_{\nu }}B \right){{\Phi }^{-1}}\left( A{{\#}_{\nu }}B \right) \right\|\le \frac{K\left( h \right)}{{{K}^{r}}\left( h' \right)}.\]
By computation, we have
\begin{align}
\nonumber & \left\| \Phi \left( A{{\nabla }_{\nu }}B \right)Mm{{K}^{r}}\left( h' \right){{\Phi }^{-1}}\left( A{{\#}_{\nu }}B \right) \right\| \\ \nonumber
& \le \frac{1}{4}{{\left\| \Phi \left( A{{\nabla }_{\nu }}B \right)+Mm{{K}^{r}}\left( h' \right){{\Phi }^{-1}}\left( A{{\#}_{\nu }}B \right) \right\|}^{2}} \quad \text{(by Lemma \ref{6})}\\ \nonumber
& \le \frac{1}{4}{{\left\| \Phi \left( A{{\nabla }_{\nu }}B \right)+Mm{{K}^{r}}\left( h' \right)\Phi \left( {{A}^{-1}}{{\#}_{\nu }}{{B}^{-1}} \right) \right\|}^{2}} \quad \text{(by Lemma \ref{7})}\\ \nonumber
& \le \frac{1}{4}{{\left\| \Phi \left( A{{\nabla }_{\nu }}B \right)+Mm\Phi \left( {{A}^{-1}}{{\nabla }_{\nu }}{{B}^{-1}} \right) \right\|}^{2}} \quad \text{(by Lemma \ref{1})}\\ 
& \le \frac{1}{4}{{\left( M+m \right)}^{2}} \quad \text{(by \eqref{4})} \label{13},
\end{align}
which leads to \eqref{10}.
\vskip 0.2 true cm
Now we prove \eqref{12}. The operator inequality \eqref{12} is equivalent to
\[\left\| \Phi \left( A{{\nabla }_{\nu }}B \right){{\left( \Phi \left( A \right){{\#}_{\nu }}\Phi \left( B \right) \right)}^{-1}} \right\|\le \frac{K\left( h \right)}{{{K}^{r}}\left( h' \right)}.\]
Compute
\[\begin{aligned}
& \left\| \Phi \left( A{{\nabla }_{\nu }}B \right)Mm{{K}^{r}}\left( h' \right){{\left( \Phi \left( A \right){{\#}_{\nu }}\Phi \left( B \right) \right)}^{-1}} \right\| \\ 
& \le \frac{1}{4}{{\left\| \Phi \left( A{{\nabla }_{\nu }}B \right)+Mm{{K}^{r}}\left( h' \right){{\left( \Phi \left( A \right){{\#}_{\nu }}\Phi \left( B \right) \right)}^{-1}} \right\|}^{2}}\quad \text{(by Lemma \ref{6})} \\ 
& \le \frac{1}{4}{{\left\| \Phi \left( A{{\nabla }_{\nu }}B \right)+Mm{{K}^{r}}\left( h' \right){{\Phi }^{-1}}\left( A{{\#}_{\nu }}B \right) \right\|}^{2}} \quad \text{(by Lemma \ref{11})}\\ 
& \le \frac{1}{4}{{\left( M+m \right)}^{2}} \quad \text{(by \eqref{13})}.\\ 
\end{aligned}\]
Thus, we complete the proof.
\end{proof}
\begin{remark}
Inequalities \eqref{21} and \eqref{22} are two special cases of Theorem \ref{5} by taking $\nu =\frac{1}{2}$. 
However, our inequalities in Theorem \ref{5} are tighter than that in \eqref{14} and \eqref{15}.
\end{remark}
\vskip 0.4 true cm
To achieve the second result, we state for easy reference the following fact obtaining from {{\cite[Theorem 3]{5}}} that will be applied below.
\begin{lemma}\label{8}
Let $A$ and $B$ be positive operators. Then 
\[\left\| {{A}^{r}}+{{B}^{r}} \right\|\le \left\| {{\left( A+B \right)}^{r}} \right\|,\] 
for each $1\le r<\infty $.
\end{lemma}
Our promised refinement of inequalities \eqref{16} and \eqref{17} can be stated as follows.
\begin{theorem}\label{20}
If $0<m\le B\le m'<M'\le A\le M$ or $0<m\le A\le m'<M'\le B\le M$, then for each $2\le p<\infty $ and $\nu \in \left[ 0,1 \right]$ we have
\begin{equation}\label{18}
{{\Phi }^{p}}\left( A{{\nabla }_{\nu }}B \right)\le {{\left( \frac{K\left( h \right)}{{{4}^{\frac{2}{p}-1}}{{K}^{r}}\left( h' \right)} \right)}^{p}}{{\Phi }^{p}}\left( A{{\#}_{\nu }}B \right),
\end{equation} 
and
\begin{equation}\label{19}
{{\Phi }^{p}}\left( A{{\nabla }_{\nu }}B \right)\le {{\left( \frac{K\left( h \right)}{{{4}^{\frac{2}{p}-1}}{{K}^{r}}\left( h' \right)} \right)}^{p}}{{\left( \Phi \left( A \right){{\#}_{\nu }}\Phi \left( B \right) \right)}^{p}},
\end{equation}
where $r=\min \left\{ \nu ,1-\nu \right\}$, $h=\frac{M}{m}$, $h'=\frac{M'}{m'}$.
\end{theorem}
\begin{proof}
It can be easily seen that the operator inequality \eqref{18} is equivalent to
\[\left\| {{\Phi }^{\frac{p}{2}}}\left( A{{\nabla }_{\nu }}B \right){{\Phi }^{-\frac{p}{2}}}\left( A{{\#}_{\nu }}B \right) \right\|\le \frac{{{\left( M+m \right)}^{p}}}{4{{M}^{\frac{p}{2}}}{{m}^{\frac{p}{2}}}{{K}^{\frac{pr}{2}}}\left( h' \right)}.\]
By simple computation
\[\begin{aligned}
& \left\| {{\Phi }^{\frac{p}{2}}}\left( A{{\nabla }_{\nu }}B \right){{M}^{\frac{p}{2}}}{{m}^{\frac{p}{2}}}{{K}^{\frac{pr}{2}}}\left( h' \right){{\Phi }^{-\frac{p}{2}}}\left( A{{\#}_{\nu }}B \right) \right\| \\ 
& \le \frac{1}{4}{{\left\| {{\Phi }^{\frac{p}{2}}}\left( A{{\nabla }_{\nu }}B \right)+{{M}^{\frac{p}{2}}}{{m}^{\frac{p}{2}}}{{K}^{\frac{pr}{2}}}\left( h' \right){{\Phi }^{-\frac{p}{2}}}\left( A{{\#}_{\nu }}B \right) \right\|}^{2}} \quad \text{(by Lemma \ref{6})}\\ 
& \le \frac{1}{4}{{\left\| {{\left( \Phi \left( A{{\nabla }_{\nu }}B \right)+Mm{{K}^{r}}\left( h' \right){{\Phi }^{-1}}\left( A{{\#}_{\nu }}B \right) \right)}^{\frac{p}{2}}} \right\|}^{2}} \quad \text{(by Lemma \ref{8})}\\ 
& = \frac{1}{4}{{\left\| \Phi \left( A{{\nabla }_{\nu }}B \right)+Mm{{K}^{r}}\left( h' \right){{\Phi }^{-1}}\left( A{{\#}_{\nu }}B \right) \right\|}^{p}} \\ 
& \le \frac{1}{4}{{\left( M+m \right)}^{p}} \quad \text{(by \eqref{13})},\\ 
\end{aligned}\]
which leads to \eqref{18}.
\vskip 0.2 true cm
The desired inequality \eqref{19} is equivalent to 
\[\left\| {{\Phi }^{\frac{p}{2}}}\left( A{{\nabla }_{\nu }}B \right){{\left( \Phi \left( A \right){{\#}_{\nu }}\Phi \left( B \right) \right)}^{-\frac{p}{2}}} \right\|\le \frac{{{\left( M+m \right)}^{p}}}{4{{M}^{\frac{p}{2}}}{{m}^{\frac{p}{2}}}{{K}^{\frac{pr}{2}}}\left( h' \right)}.\]
The result will follow from
\[\begin{aligned}
& \left\| {{\Phi }^{\frac{p}{2}}}\left( A{{\nabla }_{\nu }}B \right){{M}^{\frac{p}{2}}}{{m}^{\frac{p}{2}}}{{K}^{\frac{pr}{2}}}\left( h' \right){{\left( \Phi \left( A \right){{\#}_{\nu }}\Phi \left( B \right) \right)}^{-\frac{p}{2}}} \right\| \\ 
& \le \frac{1}{4}{{\left\| {{\Phi }^{\frac{p}{2}}}\left( A{{\nabla }_{\nu }}B \right)+{{M}^{\frac{p}{2}}}{{m}^{\frac{p}{2}}}{{K}^{\frac{pr}{2}}}\left( h' \right){{\left( \Phi \left( A \right){{\#}_{\nu }}\Phi \left( B \right) \right)}^{-\frac{p}{2}}} \right\|}^{2}} \quad \text{(by Lemma \ref{6})}\\ 
& \le \frac{1}{4}{{\left\| {{\left( \Phi \left( A{{\nabla }_{\nu }}B \right)+Mm{{K}^{r}}\left( h' \right){{\left( \Phi \left( A \right){{\#}_{\nu }}\Phi \left( B \right) \right)}^{-1}} \right)}^{\frac{p}{2}}} \right\|}^{2}} \quad \text{(by Lemma \ref{8})}\\ 
& =\frac{1}{4}{{\left\| \Phi \left( A{{\nabla }_{\nu }}B \right)+Mm{{K}^{r}}\left( h' \right){{\left( \Phi \left( A \right){{\#}_{\nu }}\Phi \left( B \right) \right)}^{-1}} \right\|}^{p}} \\ 
& \le \frac{1}{4}{{\left\| \Phi \left( A{{\nabla }_{\nu }}B \right)+Mm{{K}^{r}}\left( h' \right){{\Phi }^{-1}}\left( A{{\#}_{\nu }}B \right) \right\|}^{p}} \quad \text{(by Lemma \ref{11})}\\ 
& \le \frac{1}{4}{{\left( M+m \right)}^{p}} \quad \text{(by \eqref{13})},\\ 
\end{aligned}\]
as required. 
\end{proof}
\begin{remark}
Notice that the Kantorovich's constant $K\left( h \right)=\frac{{{\left( h+1 \right)}^{2}}}{4h},\text{ }h>0$ is an increasing function on $\left[ 1,\infty \right)$. Moreover $K\left( h \right)\ge 1$ for any $h>0$. Therefore, Theorem \ref{20} is a refinement of the inequalities, \eqref{16} and \eqref{17} for $2\le p<\infty $.
\end{remark}
\vskip 0.4 true cm
It is proved in {{\cite[Theorem 2.6]{11}}} that for $4\le p<\infty $,
\[{{\Phi }^{p}}\left( \frac{A+B}{2} \right)\le \frac{{{\left( K\left( h \right)\left( {{M}^{2}}+{{m}^{2}} \right) \right)}^{p}}}{16{{M}^{p}}{{m}^{p}}}{{\Phi }^{p}}\left( A\#B \right),\]
and
\[{{\Phi }^{p}}\left( \frac{A+B}{2} \right)\le \frac{{{\left( K\left( h \right)\left( {{M}^{2}}+{{m}^{2}} \right) \right)}^{p}}}{16{{M}^{p}}{{m}^{p}}}{{\left( \Phi \left( A \right)\#\Phi \left( B \right) \right)}^{p}}.\]
These inequalities can be improved:

\begin{theorem}
If $0<m\le B\le m'<M'\le A\le M$ or $0<m\le A\le m'<M'\le B\le M$, then for each $4\le p<\infty $ and $\nu \in \left[ 0,1 \right]$ we have
\begin{equation}\label{28}
{{\Phi }^{p}}\left( A{{\nabla }_{\nu }}B \right)\le {{\left( \frac{\sqrt{K\left( {{h}^{2}} \right)}K\left( h \right)}{{{2}^{\frac{4}{p}-1}}{{K}^{r}}\left( h' \right)} \right)}^{p}}{{\Phi }^{p}}\left( A{{\#}_{\nu }}B \right),
\end{equation}
and
\begin{equation}\label{29}
{{\Phi }^{p}}\left( A{{\nabla }_{\nu }}B \right)\le {{\left( \frac{\sqrt{K\left( {{h}^{2}} \right)}K\left( h \right)}{{{2}^{\frac{4}{p}-1}}{{K}^{r}}\left( h' \right)} \right)}^{p}}{{\left( \Phi \left( A \right){{\#}_{\nu }}\Phi \left( B \right) \right)}^{p}},
\end{equation}
where $r=\min \left\{ \nu ,1-\nu \right\}$, $K\left( h \right)=\frac{{{\left( h+1 \right)}^{2}}}{4h},K\left( h' \right)=\frac{{{\left( h'+1 \right)}^{2}}}{4h'},h=\frac{M}{m},h'=\frac{M'}{m'}$. 
\end{theorem}
\begin{proof}
It is easily verified that if $0<m\le T\le M$, then
\begin{equation}\label{63}
{{M}^{2}}{{m}^{2}}{{T}^{-2}}+{{T}^{2}}\le {{M}^{2}}+{{m}^{2}}.
\end{equation}
According to the assumption we have
	\[0<m\le A{{\nabla }_{\nu }}B\le M,\] 
we can also write
	\[0<m\le \Phi \left( A{{\nabla }_{\nu }}B \right)\le M.\] 
Using the substitution $T=\Phi \left( A{{\nabla }_{\nu }}B \right)$ in \eqref{63} we get
\begin{equation}\label{64}
{{M}^{2}}{{m}^{2}}{{\Phi }^{-2}}\left( A{{\nabla }_{\nu }}B \right)+{{\Phi }^{2}}\left( A{{\nabla }_{\nu }}B \right)\le {{M}^{2}}+{{m}^{2}}. 
\end{equation}
On the other hand, from \eqref{10} we obtain
\begin{equation}\label{26}
{{\Phi }^{-2}}\left( A{{\#}_{\nu }}B \right)\le {{\left( \frac{K\left( h \right)}{{{K}^{r}}\left( h' \right)} \right)}^{2}}{{\Phi }^{-2}}\left( A{{\nabla }_{\nu }}B \right).
\end{equation}
From this one can see that
\[\begin{aligned}
& \left\| {{\Phi }^{\frac{p}{2}}}\left( A{{\nabla }_{\nu }}B \right){{M}^{\frac{p}{2}}}{{m}^{\frac{p}{2}}}{{\Phi }^{-\frac{p}{2}}}\left( A{{\#}_{\nu }}B \right) \right\| \\ 
& \le \frac{1}{4}{{\left\| \frac{{{K}^{\frac{p}{4}}}\left( h \right)}{{{K}^{\frac{rp}{4}}}\left( h' \right)}{{\Phi }^{\frac{p}{2}}}\left( A{{\nabla }_{\nu }}B \right)+{{\left( \frac{{{K}^{r}}\left( h' \right){{M}^{2}}{{m}^{2}}}{K\left( h \right)} \right)}^{\frac{p}{4}}}{{\Phi }^{-\frac{p}{2}}}\left( A{{\#}_{\nu }}B \right) \right\|}^{2}} \quad \text{(by Lemma \ref{6})}\\ 
& \le \frac{1}{4}{{\left\| {{\left( \frac{K\left( h \right)}{{{K}^{r}}\left( h' \right)}{{\Phi }^{2}}\left( A{{\nabla }_{\nu }}B \right)+\frac{{{K}^{r}}\left( h' \right){{M}^{2}}{{m}^{2}}}{K\left( h \right)}{{\Phi }^{-2}}\left( A{{\#}_{\nu }}B \right) \right)}^{\frac{p}{4}}} \right\|}^{2}} \quad \text{(by Lemma \ref{8})}\\ 
& =\frac{1}{4}{{\left\| \frac{K\left( h \right)}{{{K}^{r}}\left( h' \right)}{{\Phi }^{2}}\left( A{{\nabla }_{\nu }}B \right)+\frac{{{K}^{r}}\left( h' \right){{M}^{2}}{{m}^{2}}}{K\left( h \right)}{{\Phi }^{-2}}\left( A{{\#}_{\nu }}B \right) \right\|}^{\frac{p}{2}}} \\ 
& \le \frac{1}{4}{{\left\| \frac{K\left( h \right)}{{{K}^{r}}\left( h' \right)}{{\Phi }^{2}}\left( A{{\nabla }_{\nu }}B \right)+\frac{K\left( h \right){{M}^{2}}{{m}^{2}}}{{{K}^{r}}\left( h' \right)}{{\Phi }^{-2}}\left( A{{\nabla }_{\nu }}B \right) \right\|}^{\frac{p}{2}}} \quad \text{(by \eqref{26})}\\ 
& =\frac{1}{4}{{\left\| \frac{K\left( h \right)}{{{K}^{r}}\left( h' \right)}\left( {{\Phi }^{2}}\left( A{{\nabla }_{\nu }}B \right)+{{M}^{2}}{{m}^{2}}{{\Phi }^{-2}}\left( A{{\nabla }_{\nu }}B \right) \right) \right\|}^{\frac{p}{2}}} \\ 
& \le \frac{1}{4}{{\left( \frac{K\left( h \right)\left( {{M}^{2}}+{{m}^{2}} \right)}{{{K}^{r}}\left( h' \right)} \right)}^{\frac{p}{2}}} \quad \text{(by \eqref{64})}.\\ 
\end{aligned}\]
Show that
\[\left\| {{\Phi }^{\frac{p}{2}}}\left( A{{\nabla }_{\nu }}B \right){{\Phi }^{-\frac{p}{2}}}\left( A{{\#}_{\nu }}B \right) \right\|\le \frac{1}{4}{{\left( \frac{K\left( h \right)\left( {{M}^{2}}+{{m}^{2}} \right)}{{{K}^{r}}\left( h' \right)Mm} \right)}^{\frac{p}{2}}}.\]
The validity of this inequality is just  inequality \eqref{28}.

Similarly, \eqref{29} holds by the inequality \eqref{12}.
\end{proof}
\begin{remark}
For $\nu =\frac{1}{2}$ such result can be found in {{\cite[Theorem 3]{3}}}.
\end{remark}

\vskip 0.3 true cm
In {{\cite[Theorem 2.1]{43}}} the authors gave operator P\'olya-Szeg\"o inequality as follows:

\begin{theorem}\label{49}
Let $\Phi $ be a positive linear map. If $0<m_{1}^{2}\le A\le M_{1}^{2}$ and $0<m_{2}^{2}\le B\le M_{2}^{2}$ for some positive real numbers ${{m}_{1}}\le {{M}_{1}}$ and ${{m}_{2}}\le {{M}_{2}}$, then 
\begin{equation}\label{48}
\Phi \left( A \right)\#\Phi \left( B \right)\le \frac{M+m}{2\sqrt{Mm}}\Phi \left( A\#B \right),
\end{equation}
where $m=\frac{{{m}_{2}}}{{{M}_{1}}}$ and $M=\frac{{{M}_{2}}}{{{m}_{1}}}$. 
\end{theorem}
It is worth noting that the inequality \eqref{48} was first proved in {{\cite[Theorem 4]{60}}} for matrices under the sandwich assumption $mA\le B\le MA$ (see also {{\cite[Theorem 3]{61}}}).
\vskip 0.3 true cm
Zhao et al. {{\cite[Theorem 3.2]{42}}} by using the same strategies of \cite{2} obtained that: 
\begin{lemma}\label{43}
If $A,B\ge 0,\text{ }1<h\le {{A}^{-\frac{1}{2}}}B{{A}^{-\frac{1}{2}}}\le h'$ or $0<h'\le {{A}^{-\frac{1}{2}}}B{{A}^{-\frac{1}{2}}}\le h<1$, then 
\begin{equation*}
K{{\left( h \right)}^{r}}A{{\#}_{\nu }}B\le A{{\nabla }_{\nu }}B,
\end{equation*}
for all $\nu \in \left[ 0,1 \right]$, where $r=\min \left\{ \nu ,1-\nu  \right\}$.  
\end{lemma}
\vskip 0.3 true cm
Now, we try to obtain a new refinement of Theorem \ref{49} by using Lemma \ref{43}.
\begin{theorem}\label{46}
Let $A,B$ be two positive operators such that $m_{1}^{2}\le A\le M_{1}^{2}$, $m_{2}^{2}\le B\le M_{2}^{2}$, $m=\frac{{{m}_{2}}}{{{M}_{1}}}$ and $M=\frac{{{M}_{2}}}{{{m}_{1}}}$. If ${{M}_{1}}<{{m}_{2}}$, then
\begin{equation}\label{47}
\Phi \left( A \right)\#\Phi \left( B \right)\le \gamma \Phi \left( A\#B \right),
\end{equation}
where
\[\gamma =\frac{M+m}{2\sqrt{MmK\left( h \right)}},\]
and $h=\frac{m_{2}^{2}}{M_{1}^{2}}$.

Moreover the inequality \eqref{47} holds true for ${{M}_{2}}<{{m}_{1}}$ and $h=\frac{M_{2}^{2}}{m_{1}^{2}}$.
\end{theorem}
\begin{proof}
According to the assumptions we have
\[{{m}^{2}}=\frac{m_{2}^{2}}{M_{1}^{2}}\le {{A}^{-\frac{1}{2}}}B{{A}^{-\frac{1}{2}}}\le \frac{M_{2}^{2}}{m_{1}^{2}}={{M}^{2}},\]
i.e.,
\begin{equation}\label{41}
m\le {{\left( {{A}^{-\frac{1}{2}}}B{{A}^{-\frac{1}{2}}} \right)}^{\frac{1}{2}}}\le M.
\end{equation}
Therefore \eqref{41} implies that
\[\left( {{\left( {{A}^{-\frac{1}{2}}}B{{A}^{-\frac{1}{2}}} \right)}^{\frac{1}{2}}}-m \right)\left( M-{{\left( {{A}^{-\frac{1}{2}}}B{{A}^{-\frac{1}{2}}} \right)}^{\frac{1}{2}}} \right)\ge 0.\]
Simplifying we find that
\[\left( M+m \right){{\left( {{A}^{-\frac{1}{2}}}B{{A}^{-\frac{1}{2}}} \right)}^{\frac{1}{2}}}\ge Mm+{{A}^{-\frac{1}{2}}}B{{A}^{-\frac{1}{2}}}.\]
Multiplying both sides by ${{A}^{\frac{1}{2}}}$ to get
\begin{equation}\label{42}
\left( M+m \right)A\#B\ge MmA+B.
\end{equation}
By applying positive linear map $\Phi $ in \eqref{42} we infer
\begin{equation}\label{45}
\left( M+m \right)\Phi \left( A\#B \right)\ge Mm\Phi \left( A \right)+\Phi \left( B \right).
\end{equation}
Utilizing Lemma \ref{43} for the case $\nu =\frac{1}{2},\text{ }A=Mm\Phi \left( A \right)$ and $B=\Phi \left( B \right)$, and by taking into account that $m_{1}^{2}\frac{{{M}_{2}}{{m}_{2}}}{{{M}_{1}}{{m}_{1}}}\le \frac{{{M}_{2}}{{m}_{2}}}{{{M}_{1}}{{m}_{1}}}\Phi \left( A \right)\le M_{1}^{2}\frac{{{M}_{2}}{{m}_{2}}}{{{M}_{1}}{{m}_{1}}}$ implies  $m_{1}^{2}\le \Phi \left( A \right)\le M_{1}^{2}$, and $1<\frac{m_{2}^{2}}{M_{1}^{2}}\le {{\Phi }^{-\frac{1}{2}}}\left( A \right)\Phi \left( B \right){{\Phi }^{-\frac{1}{2}}}\left( A \right)\le \frac{M_{2}^{2}}{m_{1}^{2}}$ we infer
\begin{equation}\label{44}
2\sqrt{MmK\left( h \right)}\Phi \left( A \right)\#\Phi \left( B \right)\le Mm\Phi \left( A \right)+\Phi \left( B \right),
\end{equation}
where $h=\frac{m_{2}^{2}}{M_{1}^{2}}$.

Combining \eqref{45} and \eqref{44}, we deduce the desired result \eqref{47}. The case ${{M}_{2}}<{{m}_{1}}$ is similar, we omit the details.

This completes the proof.
\end{proof}
\vskip 0.3 true cm
The authors would like to pose the following question that is interesting on its own right.
\begin{question}
Is the constant $\gamma $ in Theorem \ref{46} sharp?
\end{question}
\begin{example}
Taking $m_{1}^{2}=1.21,\text{ }M_{1}^{2}=16,\text{ }m_{2}^{2}=20.25,\text{ }M_{2}^{2}=25,\text{ }A=\left( \begin{matrix}
   2 & -2  \\
   -2 & 7  \\
\end{matrix} \right),\text{ }B=\left( \begin{matrix}
   21 & 0.5  \\
   0.5 & 21  \\
\end{matrix} \right)$
and $\Phi \left( X \right)=\frac{1}{2}tr\left( X \right)\left( X\in {{\mathcal{M}}_{2}} \right)$, by an easy computation we find that
\[\Phi \left( A \right)\#\Phi \left( B \right)\simeq 9.72,\]
\[\frac{M+m}{2\sqrt{Mm}}\Phi \left( A\#B \right)\simeq 11.2,\]
and
\[\frac{M+m}{2\sqrt{MmK\left( h \right)}}\Phi \left( A\#B \right)\simeq 11.12,\]
which shows that if ${{M}_{1}}<{{m}_{2}}$, then inequality \eqref{47} is really an improvement of \eqref{48}. 
\end{example}
\begin{example}
Assume that $m_{1}^{2}=4,\text{ }M_{1}^{2}=9,\text{ }m_{2}^{2}=0.5,\text{ }M_{2}^{2}=2,\text{ }A=\left( \begin{matrix}
   6 & -1  \\
   -1 & 5  \\
\end{matrix} \right),\text{ }B=\left( \begin{matrix}
   1.5 & 0.5  \\
   0.5 & 1.2  \\
\end{matrix} \right)$
and $\Phi \left( X \right)=\frac{1}{2}tr\left( X \right)\left( X\in {{\mathcal{M}}_{2}} \right)$, by an easy computation we find that
\[\Phi \left( A \right)\#\Phi \left( B \right)\simeq 2.72,\]
\[\frac{M+m}{2\sqrt{Mm}}\Phi \left( A\#B \right)\simeq 3.02,\]
and
\[\frac{M+m}{2\sqrt{MmK\left( h \right)}}\Phi \left( A\#B \right)\simeq 2.84,\]
which shows that if ${{M}_{2}}<{{m}_{1}}$, then inequality \eqref{47} is really an improvement of \eqref{48}.
\end{example}
\vskip 0.3 true cm
The inequality \eqref{47} can be squared by a similar method as in {{\cite[Theorem 2.3]{73}}}:
\begin{theorem}\label{97}
Suppose all the assumptions of Theorem \ref{46} be satisfied. Then
\begin{equation}\label{61}
{{\left( \Phi \left( A \right)\#\Phi \left( B \right) \right)}^{2}}\le \psi {{\Phi }^{2}}\left( A\#B \right),
\end{equation}
where 
\begin{equation*}
\psi =\left\{ \begin{array}{lr}
   \frac{{{\gamma }^{2}}{{\left( \alpha +\beta  \right)}^{2}}}{4\alpha \beta }&\text{ if }\alpha \le {{t}_{0}} \\ 
  \frac{\gamma \left( \alpha +\beta  \right)-\beta }{\alpha }&\text{ if }\alpha \ge {{t}_{0}} \\ 
\end{array} \right.,
\end{equation*}
$\alpha ={{m}_{1}}{{m}_{2}}$, and $\beta ={{M}_{1}}{{M}_{2}}$.
\end{theorem}
\begin{proof}
According to the assumption one can see that
\begin{equation}\label{58}
\alpha \le \Phi \left( A\#B \right)\le \beta,
\end{equation}	 
and
\begin{equation}\label{59}
\alpha \le \Phi \left( A \right)\#\Phi \left( B \right)\le \beta 
\end{equation}	
where 
\[\alpha ={{m}_{1}}{{m}_{2}},\quad \beta ={{M}_{1}}{{M}_{2}}.\]
Inequality \eqref{58} implies
\[{{\Phi }^{2}}\left( A\#B \right)\le \left( \alpha +\beta  \right)\Phi \left( A\#B \right)-\alpha \beta, \]
and \eqref{59} give us
\begin{equation}\label{60}
{{\left( \Phi \left( A \right)\#\Phi \left( B \right) \right)}^{2}}\le \left( \alpha +\beta  \right)\left( \Phi \left( A \right)\#\Phi \left( B \right) \right)-\alpha \beta .
\end{equation}

Hence
\begin{align}
 \nonumber & {{\Phi }^{-1}}\left( A\#B \right){{\left( \Phi \left( A \right)\#\Phi \left( A \right) \right)}^{2}}{{\Phi }^{-1}}\left( A\#B \right) \\ \nonumber
 & \le {{\Phi }^{-1}}\left( A\#B \right)\left( \left( \alpha +\beta  \right)\left(\Phi \left( A \right)\#\Phi \left( B \right)\right)-\alpha \beta  \right){{\Phi }^{-1}}\left( A\#B \right) \quad \text{(by \eqref{60})}\\ \label{62}
 & \le \left( \gamma \left( \alpha +\beta  \right)\Phi \left( A\#B \right)-\alpha \beta  \right){{\Phi }^{-2}}\left( A\#B \right) \quad \text{(by \eqref{47})}.\\ \nonumber
\end{align}

Consider the real function $f\left( t \right)$ on $\left( 0,\infty  \right)$ defined as
\[f\left( t \right)=\frac{\gamma \left( \alpha +\beta  \right)t-\alpha \beta }{{{t}^{2}}}.\]
As a matter of fact, the inequality \eqref{62} implies that
\[{{\Phi }^{-1}}\left( A\#B \right){{\left( \Phi \left( A \right)\#\Phi \left( B \right) \right)}^{2}}{{\Phi }^{-1}}\left( A\#B \right)\le \underset{\alpha \le t\le \beta }{\mathop{\max }}\,f\left( t \right).\]
One can see that the function $f\left( t \right)$ is decreasing on $\left[ \alpha ,\beta  \right]$. By an easy computation we have 
\[f'\left( t \right)=\frac{2\alpha \beta -\gamma \left( \alpha +\beta  \right)t}{{{t}^{3}}}.\]
This function has an maximum point on 
\[{{t}_{0}}=\frac{2\alpha \beta }{\gamma \left( \alpha +\beta  \right)},\]
with the maximum value
\[f\left( {{t}_{0}} \right)=\frac{{{\gamma }^{2}}{{\left( \alpha +\beta  \right)}^{2}}}{4\alpha \beta }.\]
Whence
\begin{equation*}
\underset{\alpha \le t\le \beta }{\mathop{\max }}\,f\left( t \right)\le \left\{ \begin{array}{lr}
   f\left( {{t}_{0}} \right)&\text{ for }\alpha\le {{t}_{0}} \\ 
  f\left( \alpha  \right)&\text{ for }\alpha\ge {{t}_{0}} \\ 
\end{array} \right..
\end{equation*}
Notice that
\[f\left( \alpha  \right)=\frac{\gamma \left( \alpha +\beta  \right)-\beta }{\alpha }.\]

It is striking that we can get the same inequality \eqref{61} under the condition ${{M}_{2}}<{{m}_{1}}$. 

Hence the proof of Theorem \ref{97} is complete. 
\end{proof}
\vskip 0.5 true cm
\noindent {\bf Acknowledgment.} The authors are deeply indebted to Professor Jean Christophe Bourin  for calling our attention to the work of Lee \cite{60}.  The authors
also thank the referee for his useful comments which improved the current paper.


\vskip 0.4 true cm

{\tiny $^1$Department of Mathematics, Mashhad Branch, Islamic Azad University, Mashhad, Iran.

{\it E-mail address:} hrmoradi@mshdiau.ac.ir

\vskip 0.4 true cm

{\tiny $^2$Department of Mathematics, Mashhad Branch, Islamic Azad University, Mashhad, Iran.

{\it E-mail address:} erfanian@mshdiau.ac.ir

\vskip 0.4 true cm

$^3$Department of Mathematics, Faculty of Arts and Sciences, Adiyaman University, Adiyaman, Turkey.

{\it E-mail address:} igumus@adiyaman.edu.tr}

\vskip 0.4 true cm

$^4$Department of Mathematics, Payame Noor University, P.O. Box 19395-3697 Tehran, Iran.

{\it E-mail address:} raziyehnaseri29@gmail.com}
\end{document}